\documentclass[a4, 12pt]{amsart}
\usepackage{a4wide,amssymb}
\usepackage{yfonts}
\usepackage{hyperref}
 \usepackage{tikz}
  \usetikzlibrary{matrix,arrows,decorations.pathmorphing}
 
\renewcommand{\subset}{\subseteq}

\newtheorem*{remark*}{Remark}

\newtheorem{proposition}{Proposition}[section]
\newtheorem{corollary}{Corollary}[section]
\newtheorem{lemma}{Lemma}[section]

\newtheorem{theorem}{Theorem}[section]
\newtheorem*{theorem*}{Theorem}
\newtheorem{statement*}{Statement}

\makeatletter

\title{Duflo's conjecture for the branching to the Iwasawa $AN$-subgroup }
\author[Liu]{Gang Liu}
\address{G. Liu, Leibniz Universit\"{a}t Hannover, Institut f\"{u}r Analysis
, Welfengarten 1, 30167 Hannover Germany.}
\email{liu@math.uni-hannover.de}

\begin{document}
%\large
%
%\def\section
%{\@startsection {section}{1}
%{-1pt}{-5ex \@plus -1ex \@minus -.2ex}
%{2ex \@plus .2ex}
%{\normalfont \large \bfseries } }

\maketitle

\begin{abstract}
The purpose of this paper is to prove Duflo's conjecture for $(G,\pi, AN)$ where $G$ is a simple Lie group of Hermitian type  and $\pi$ is a discrete series of $G$ and $AN$ is the maximal exponential solvable subgroup for an Iwasawa decomposition $G=KAN$. This is essentially reduced from the following general theorem we prove in this paper:  let $G$ be a connected semisimple Lie group . Then a strongly elliptic $G$-coadjoint orbit $\mathcal{O}$ is holomorphic if and only if  $\text{p}(\mathcal{O})$ is an open $AN$-coadjoint orbit, where $\text{p} : \mathfrak{g}^* \longrightarrow (\mathfrak{a}\oplus\mathfrak{n})^*$ is the natural projection. \end{abstract}

\section{Introduction}

Let $H\subset G$ be  real connected Lie groups of type \textrm{I} with Lie algebras $\mathfrak{h}\subset\mathfrak{g}$. Let $\pi$ be a unitary irreducible representation of $G$. One fundamental problem in representation theory and harmonic analysis is to study the restriction of $\pi$ to $H$ i.e., the branching problem. For $G$ exponential solvable, the branching problem was determined in (\cite{Fu}). However, it is very hard to find an explicit branching laws for general $(G,\pi, H)$, especially for $G$ reductive and $H$ non reductive. When $G, H$ are both reductive, good progress has been made, notably by  work of Kobayashi (\cite{Ko1}, \cite{Ko2}, \cite{Ko3}, \cite{KOP}) and recent work of Duflo-Vargas (\cite{DV1}, \cite{DV2}).

A central problem in branching theory, initiated by Kobayashi, is to study when $\pi\vert_{H}$ is $H$-admissible (in the sense of Kobayashi): i.e., $\pi\vert_{H}$ is decomposed discretely with finite multiplicities.

Now let us consider the branching problem geometrically. Suppose that $\pi$ is attached to a $G$-coadjoint orbit $\mathcal{O}$ in $\mathfrak{g}^*$: i.e., $\pi$ is a "quantization" of $\mathcal{O}$. Then $\mathcal{O}$, equipped with the Kirillov-Kostant-Souriau symplectic form $\varpi$, becomes a $H$-Hamiltonian space. The corresponding moment map is just the natural projection $\text{p}: \mathcal{O}\longrightarrow \mathfrak{h}^*$. One might care whether the branching law $\pi\vert_{H}$ can be studied via the $H$-Hamiltonian space $(\mathcal{O}, \varpi)$.

The answer is positive for $G$ exponential solvable (by the work of Fujiwara \cite{Fu}) or compact (by work of Heckman \cite{He} and Guillemin-Sternberg \cite{GS}). But for general $G$, the answer is not that clear : for instance not all $\pi\in\widehat{H}$ can be associated with a coadjoint orbit. Next even if $\pi$ is attached to a certain orbit $\mathcal{O}$, it is not clear that each $H$-irreducible representation which appears in $\pi\vert_{H}$ can be attached to a $H$-coadjoint orbit in $\mathfrak{h}^*$. Nevertheless, for $H\subset G$ which are almost algebraic Lie groups and $\pi$ is a \textit{discrete series} of $G$,  Duflo proposes a conjecture which relates the branching problem to the geometry of the moment map.

 For $G$ semisimple and $\pi$ a discrete series of $G$ which is attached to a $G$-coadjoint orbit $\mathcal{O}_{\pi}$ (in the sense of Duflo), Duflo's conjecture states as follows:\\

\noindent (D1) $\pi\vert_{H}$ is $H$-admissible if and only if the projection $\text{p}$ is \textit{weakly proper}.

 \noindent (D2) If $\pi\vert_{H}$ is $H$-admissible, then each irreducible $H$-representation $\sigma$ which appears in $\pi\vert_{H}$ is attached to a \textit{strongly regular} $H$-coadjoint orbit $\Omega$ (in the sense of Duflo) which is contained in $\text{p}(\mathcal{O}_{\pi})$.

 \noindent (D3) If $\pi\vert_{H}$ is $H$-admissible, the multiplicity of each such $\sigma$  can be expressed geometrically on the reduced space of $\Omega$ (with respect to the moment map $\text{p}$).\\

  Here "weakly proper" in (D1) means that the preimage (for $\text{p}$) of each compact subset which is contained in $\text{p}(\mathcal{O}_{\pi})\cap\Upsilon_{sr}$ is compact in $\mathcal{O}_{\pi}$. Here $\Upsilon_{sr}$ is the set of all strongly regular elements in  $\mathfrak{h}^*$. For the definition of strongly regular elements (orbits) and more information on Duflo's conjecture, we refer to (\cite{Liu}).

   As for (D2), let us remark that in the framework of Duflo's orbit method, each discrete series of $G$ (resp. $H$) is attached to a strongly regular $G$ (resp. $H$)-coadjoint orbit. Moreover according to Duflo-Vargas's work (\cite{DV1}, \cite{DV2}), each irreducible $H$-representation $\widetilde{\sigma}$ which appears in the integral decomposition of  $\pi\vert_{H}$ (which is not necessarily $H$-admissible) is attached to a strongly regular $H$-coadjoint orbit. Note that $\widetilde{\sigma}$ is not necessarily a discrete series. However, if $\pi\vert_{H}$ is $H$-admissible, then each $H$-irreducible representation appearing in $\pi\vert_{H}$ must be a discrete series. Thus  (D2) has a good geometric meaning.

  (D3) seems not very explicit, but at least it suggests a direction. It is obviously influenced by Guillemin-Sternberg's work for compact Hamiltonian spaces and especially encouraged by Paradan's work on non-compact Hamiltonian spaces. We will see in our case, this direction is also correct.

 In this paper, we will prove Duflo's conjecture for $(G, AN)$ where $G$ is a simple Lie group of Hermitian type  and $AN$ is the maximal exponential solvable subgroup for an Iwasawa decomposition $G=KAN$.

 The paper is organized as follows. In section 2, assuming the results in  section 5, we prove Duflo's conjecture for the case mentioned above.  In section 3, we state a general geometric theorem (Theorem 3.1) on strongly elliptic coadjoint orbits.  The sections 4 and 5 are devoted to proving Theorem 3.1.\\

\noindent \textbf{Notations and Conventions}: $\mathbb{R}^{*}_+=\{x: x>0\}$,  $\mathbb{R}^{*}_-=\{x: x<0\}$. For any Lie algebra $\mathfrak{g}$, $\mathfrak{g}^*$ denotes its algebraic dual.

\section{Duflo's conjecture for simple Lie groups of Hermitian type }

Let $G=KAN$ be an Iwasawa decomposition of a connected semisimple Lie group $G$. Let $\mathfrak{a}$ (resp. $\mathfrak{n}$) be the Lie algebra of $A$ (resp. $N$). Let $\pi$ be a discrete series of $G$ with $\mathcal{O}_{\pi}$ its associated coadjoint orbit (in the sense of Duflo). Suppose that there exists an open $AN$-coadjoint orbit in $(\mathfrak{a}\oplus\mathfrak{n})^*$. Then in Th\'{e}or\`{e}me 5.3 of \cite{Liu}, we proved

\begin{theorem}
   The projection $\text{p}: \mathcal{O}_{\pi}\longrightarrow (\mathfrak{a}\oplus\mathfrak{n})^*$ is weakly proper if and only if $\text{p}(\mathcal{O}_{\pi})$ is an open $AN$-coadjoint orbit.
\end{theorem}

Now suppose that $G$ is simple of Hermitian type. In this section, we will prove Duflo's conjecture for $(G, AN)$. It is well known that for such Lie group $G$, there exists an open $AN$-coadjoint orbit in $(\mathfrak{a}\oplus\mathfrak{n})^*$.

On the other hand it is known that $\pi\vert_{AN}$ is $AN$-admissible if and only if $\pi$ is holomorphic or anti-holomorphic (see \cite{Riv} and the Theorem 4.6 of \cite{Rev}).

Hence (D1) of Duflo's conjecture follows from Theorem 2.1 and Theorem 3.1 below which states that $\text{p}(\mathcal{O}_{\pi})$  is an open $AN$-coadjoint orbit if and only if $\pi$ is holomorphic or anti-holomorphic.

Next let us check (D2). Let $G=K\exp{\mathfrak{p}}$ be the associated Cartan decomposition of $G$ and $\mathfrak{t}$ be a Cartan sub-algebra for $K$. Let $\pi$ be a holomorphic discrete series of $G$. Without loss of generality, we can suppose that its Harish-Chandra parameter $\lambda$ is in $(i\mathfrak{t})^*$. Let $f=-i\lambda\in \mathfrak{t}^*\subset \mathfrak{g}^*$ (the relation $\mathfrak{t}^*\subset \mathfrak{g}^*$ is relative to the decomposition $\mathfrak{g}=\mathfrak{t}\oplus [\mathfrak{t},\mathfrak{g}]$). Then $\mathcal{O}_{\pi}=G.f$ is the associated coadjoint orbit of $\pi$ (in the sense of Duflo).
Let $\Delta$ be the root system with respect to $(\mathfrak{g}_{\mathbb{C}},\mathfrak{t}_{\mathbb{C}})$. We identify $ \mathfrak{t}^*$ with $\mathfrak{t}$ under the inner product $\langle.,.\rangle:=-\mathcal{K}(.,.)_{\mathfrak{t}\times\mathfrak{t}}$, where $\mathcal{K}(.,.)$ is the Killing form on $\mathfrak{g}$.
 Then $\Delta^+:=\{\alpha\in \Delta: \langle f, i\alpha \rangle> 0\}$ is a subset of positive roots. Let $\Delta_{n}^+\subset\Delta^+$ (resp. $\Delta_{c}^+\subset\Delta^+$) be the subset of positive non-compact (resp. compact) roots. Then $\Delta_{n}^+$ determines the holomorphic structure $\mathfrak{p}^+_{\mathbb{C}}$ in $\mathfrak{p}_{\mathbb{C}}$. Hence there exists a unique element $Z_0$ in the center of $\mathfrak{k}$ such that $[Z_0, X]= -iX$ for all $X\in\mathfrak{p}^+_{\mathbb{C}}$. Then $ \langle Z_0, i\alpha \rangle=1> 0$ for all $\alpha\in \Delta_{n}^+$. Let $h\in (\mathfrak{a}\oplus\mathfrak{n})^*$  such that $h(Y)=-\mathcal{K}(Z_0, Y)$ for all $Y\in\mathfrak{a}\oplus\mathfrak{n}$.  Then it is known that $h$ lies in an open $AN$-coadjoint orbit $\Omega$ (see for instance \cite{Riv}).

 Let $\Lambda=\lambda+\rho_G-2\rho_{K}$ be the Blattner parameter of $\pi$, where $\rho_G$ (resp. $\rho_{K}$) is the half sum of positive roots (resp. compact positive roots). Let $\pi_{\Omega}\in \widehat{AN}$  be the associated unitary irreducible representation of $\Omega=AN.h$. Let $\tau_{\Lambda}\in \widehat{K}$ be the $K$-unitary irreducible representation whose highest weight is $\Lambda$ (with respect to $\Delta_{c}^+$). Then we have the following theorem due to Rossi-Vergne (see \cite{Riv}).

\begin{theorem}
 $\pi\vert_{AN}=\dim(\tau_{\Lambda}).\pi_{\Omega}$.
\end{theorem}

 However, according to section 5, we have $\text{p}(\mathcal{O}_{\pi})=\Omega$. This is also directly deduced  from  Theorem 3.1 and Carmona's results (see section 4). For $\pi$ anti-holomorphic, it is treated exactly in the same way. Hence (D2) of Duflo's conjecture is true. Note that in our case, strongly regular $AN$-coadjoint orbits are nothing else but open orbits.

 Below we will prove (D3) of Duflo's conjecture.

\subsection{Reduced space and multiplicity }

Let $\tau_{\Lambda'}\in\widehat{K}$ be the unitary irreducible representation of $K$ with Harish-Chandra parameter $\lambda$ (with respect to $\Delta_{c}^+$). Then the highest weight of $\tau_{\Lambda'}$, $\Lambda'=\lambda-\rho_{K}$.  As ${\Lambda}-\Lambda'=\rho_n$ which is a character of $\exp(\mathfrak{t})$, we have  \\

\noindent \textbf{Observation}: $\dim(\tau_{\Lambda})=\dim(\tau_{\Lambda'})$.\\

Let $\varpi$ be the Kirillov-Kostant-Souriau form of $\mathcal{O}_{\pi}$ and $X_{\Omega}$ be the reduced space of  the open $AN$-orbit $\Omega=\text{p}(\mathcal{O}_{\pi})$. Since $AN$ is diffeomorphic to $\Omega$, we deduce that $X_{\Omega}$ is diffeomorphic to $K.f$. Then in particular $X_{\Omega}$ is a compact symplectic sub-manifold of $\mathcal{O}_{\pi}$.  Denote by $\varpi_{\Omega}$ the induced symplectic form of $X_{\Omega}$ (from $\varpi$), and $\beta_{\Omega}:=\frac{\varpi_{\Omega}^l}{(2\pi)^l.(l)!}$ the associated Liouville volume. Here $l=\dim{X_{\Omega}}=\dim{K.f}$.

Now we will prove the following theorem.

\begin{theorem}
 $\pi\vert_{AN}=(\int_{X_{\Omega}}\beta_{\Omega}).\pi_{\Omega}$.
\end{theorem}

\begin{proof}
According to the theorem of Rossi-Vergne (Theorem 2.2 above) and the observation above, it is sufficient to prove $\int_{X_{\Omega}}\beta_{\Omega}=\dim(\tau_{\Lambda'})$. Without loss of generality, we can assume that $f$ is also integral (i.e. there exists a unitary character $\chi_f$ of $G(f)=\mathbb{T}:=\exp{\mathfrak{t}}$, such that $d\chi_f=if$). Because otherwise, we can always choose a good covering $\widetilde{G}$ of $G$ such that $f$ is integral for $\widetilde{G}$ (and of course coverings do not change anything about multiplicities). As in (\cite{Liu})for $SU(2,1)$ case, we can deduce that $\int_{X_{\Omega}}\beta_{\Omega}=\int_{K.f} \beta_{K}$, where $\beta_{K}$ is the Liouville volume of $K.f$ for the induced symplectic form $\varpi_K$ on $K.f$ (from the symplectic form $\varpi$ of $\mathcal{O}_{\pi}$). However it is clear that $(K.f, \varpi_K)$ is isomorphic to the $K$-coadjoint orbit on $(K.f_K,\varpi_{f_K})$,  where $f_K=f\vert_{\mathfrak{k}}\in\mathfrak{k}^*$ and $\varpi_{f_K}$ is the Kirillov-Kostant-Souriau symplectic form on $K.f_K$. Thus it is clear that $\int_{K.f} \beta_{K}=\int_{K.f_K} \beta_{f_K}$, where $\beta_{f_K}$ is the Liouville volume for $(K.f_K,\varpi_{f_K})$. On the other hand, it is clear that the associated irreducible unitary representation for $(K.f_K,\varpi_{f_K})$ is exactly $\tau_{\Lambda'}$. Hence according to \textit{Kirillov-Rossmann's formula} (see \cite{DHV}), $\int_{K.f_K} \beta_{f_K}=\dim(\tau_{\Lambda'})$. Thus the theorem is proved.
\end{proof}

So the assertion (iii) is true according to the the above theorem. \\

\noindent \textbf{Remark}. From the previous theorem, we see that the $AN$-multiplicity equals a natural integral on the reduced space. However, it also equals "very probably" an equivariant \textit{$\text{Spin}_c$-index} on the reduced space which is reduced from the \textit{$\text{Spin}_c$-quantization} of the $G$-orbit $\mathcal{O}_{\pi}$. This equivariant index is the so-called \textit{reduction}. In other words, in this situation, the principle \textit{quantization commutes with reduction} holds. Hence this geometric principle is extended to Hamiltonian action of non-reductive Lie groups.

\section{A geometric theorem for strongly elliptic coadjoint orbits}

Let $G$ be a real connected semisimple Lie group, $\mathfrak{g}=\text{Lie}(G)$. We let $G$ act on $\mathfrak{g}^*$ by coadjoint action. Recall that an element $f\in \mathfrak{g}^*$ is called \textit{strongly elliptic}, if the Lie algebra of its stabilizer, $\mathfrak{g}(f)$ is  compact.

Now let $f$ be strongly elliptic. Then $\mathfrak{g}(f)$ contains a compact Cartan sub-algebra $\mathfrak{t}$ (conversely, if $\mathfrak{g}$ has a compact Cartan sub-algebra, then the set of strongly elliptic elements is not empty).  Since $\mathfrak{g}=\mathfrak{t}\oplus[\mathfrak{t}, \mathfrak{g}]$, and $f$ vanishes on $[\mathfrak{t}, \mathfrak{g}]$, we can  regard $f\in \mathfrak{t}^*$.  Let $\Delta$ be the root system with respect to $(\mathfrak{g}_{\mathbb{C}},\mathfrak{t}_{\mathbb{C}})$ and $G=K\exp(\mathfrak{p})$ be the associated Cartan decomposition. Let $\Delta_c$ (resp. $\Delta_n$) be the subset of compact (resp. noncompact) roots of $\Delta$. It is not hard to see that for each $\alpha\in\Delta_n$, we have $ \langle f, i\alpha\rangle\neq0$. Define the subset $\Delta_n^+=\{\alpha\in\Delta_n: \langle f, i\alpha\rangle>0\}$  where $\langle, \rangle$ is the inner product over $\mathfrak{t}^*\cong\mathfrak{t}$ deduced from the Killing form defined in section 2.

We say a strongly elliptic element $f\in\mathfrak{g}^*$ is \textit{holomorphic}, if $\Sigma_{\alpha\in \Delta_{n}^+} \mathfrak{g}_\alpha$ is a (abelian) sub-algebra of $\mathfrak{p}_{\mathbb{C}}$. Here $\mathfrak{g}_\alpha$ is the root space of $\alpha$. Then it is well known that $f$ is holomorphic if and only if $\Delta_{n}^+$ is stable under the compact Weyl group $W_K$. Notice that the existence of a strongly elliptic and holomorphic element implies that $\mathfrak{g}$ is of Hermitian type.

  A coadjoint orbit $\mathcal{O}$ (in $\mathfrak{g}^*$) is called strongly elliptic if an element (then each element) in $\mathcal{O}$ is strongly elliptic. A strongly regular orbit is called holomorphic, if an element (then each element) in it is holomorphic. Note that the subset of strongly elliptic (resp. strongly elliptic and holomorphic) elements is a $G$-invariant cone, if it is non-empty.

In the framework of orbit method, each discrete series $\pi$ of $G$ is associated to  a (unique) coadjoint orbit $\mathcal{O}$ which is regular and strongly elliptic (in the sense of Duflo). Moreover $\pi$ is holomorphic if and only if $\mathcal{O}$ is holomorphic. Note that a regular and strongly elliptic coadjoint orbit is strongly regular.

Our goal is to prove the following Theorem:

%{\parindent 0pt
%{\bf Theorem $\mathbf 1$:}
\begin{theorem} \label{holo}
 Let $G=KAN$ be an Iwasawa decomposition of a connected semisimple Lie group $G$ with Lie algebra $\mathfrak{g}$.  Let $\mathfrak{a}=\text{Lie}(A)$ and $\mathfrak{n}=\text{Lie}(N)$. Let $\text{p} : \mathfrak{g}^* \longrightarrow (\mathfrak{a}\oplus\mathfrak{n})^*$ be the natural projection. Assume that $f\in\mathfrak{g}^*$ is a strongly elliptic element with coadjoint orbit $\mathcal{O}_{f}:=G.f$. Then $f$ is holomorphic if and only if $\text{p}(\mathcal{O}_{f})$ is an open $AN$-coadjoint orbit in $(\mathfrak{a}\oplus\mathfrak{n})^*$.
\end{theorem}

 \noindent \textbf{Remark}. (1) If there exists no open $AN$-coadjoint orbit in $(\mathfrak{a}\oplus\mathfrak{n})^*$, then it is clear that there is no holomorphic element in $\mathfrak{g}^*$. Thus in this case Theorem 3.1 is true. Hence in order to prove Theorem 3.1, we can always assume the existence of an open $AN$-coadjoint orbit in $(\mathfrak{a}\oplus\mathfrak{n})^*$.

 (2) As we mentioned previously, for all semisimple Lie groups $G=KAN$ of Hermitian type, there exists an open $AN$-coadjoint orbit in $(\mathfrak{a}\oplus\mathfrak{n})^*$. However, there are also other semisimple Lie groups $G$ of non-Hermitian type, for which there exists an open $AN$-coadjoint orbit in $(\mathfrak{a}\oplus\mathfrak{n})^*$: such as the connected non compact Lie group $G$ whose Lie algebra $\mathfrak{g}$ is the split real form of the simple complex Lie algebra of type $\textrm{G}_2$.

(3) It is clear that the theorem (and the proof of the theorem) is independent of any choice of the Cartan decomposition and the subgroup $AN$.  In the extreme situation where $AN$ is reduced to a point (in other words $G$ is semisimple compact), it is clear that the theorem is correct. Thus in the following sections, we suppose that $AN$ is not trivial (i.e. $G$ is not compact).

\section{Characterization of open $AN$-coadjoint orbits in $(\mathfrak{a}\oplus\mathfrak{n})^*$}

In this section, we will give some results on open $AN$-orbits, which are essential for our proof of theorem 3.1. All these results can be found in (\cite{Ca}).

Let $G=KAN$ be an Iwasawa decomposition for a semisimple Lie group $G$. Denote $\mathfrak{a}:=\text{Lie}(A)$ and $\mathfrak{n}:=\text{Lie}(N)$. Notice that a priori, we do not assume there exists an open $AN$-coadjoint orbit in $(\mathfrak{a}\oplus\mathfrak{n})^*$.  Let $\mathfrak{h}=\mathfrak{h}_{k}\oplus \mathfrak{a}$ be a $\theta$-stable Cartan sub-algebra containing $\mathfrak{a}$. Denote by $\Phi_{\mathfrak{a}}$ (resp. $\Phi_{\mathfrak{h}}$) the system of restricted roots (resp. roots) with respect to $(\mathfrak{g}, \mathfrak{a})$ (resp. $(\mathfrak{g}_{\mathbb{C}}, \mathfrak{h}_{\mathbb{C}})$).

We can choose a set of positive roots $\Phi^{+}_{\mathfrak{a}}$ (resp. $\Phi^{+}_{\mathfrak{h}}$) for $\Phi_{\mathfrak{a}}$ (resp. $\Phi_{\mathfrak{h}}$) such that the elements in $\Phi^{+}_{\mathfrak{a}}$ are the restrictions to $\mathfrak{a}$ of the elements of $\Phi^{+}_{\mathfrak{h}}$ which are non zero over $\mathfrak{a}$. Then starting with the highest root $\beta_1$ of $\Phi^{+}_{\mathfrak{a}}$, we can construct a particular maximal strongly orthogonal set $\Upsilon=\{\beta_j\}_{1\leq j\leq r}\subset\Phi^{+}_{\mathfrak{a}}$.

Then there exists an open $AN$-coadjoint orbit in $(\mathfrak{a}\oplus\mathfrak{n})^*$ if and only if $r=\dim(\mathfrak{a})$ and $\dim(\mathfrak{g}_{\beta_j})=1$ for all $1\leq j\leq r$, where $\mathfrak{g}_{\beta_j}$ is the restricted root space of $\beta_j$. It is clear that in this case, each $\beta_j$ is the restriction of a unique real root (i.e vanishes on $\mathfrak{h}_{k}$) in $\Phi^{+}_{\mathfrak{h}}$.

Now assume that there exists an open $AN$-coadjoint orbit in $(\mathfrak{a}\oplus\mathfrak{n})^*$. Fix a non-zero element $X_j\in \mathfrak{g}_{\beta_j}$ for each $1\leq j\leq r$. Then we can find in each open $AN$-coadjoint orbit a unique element $s$, such that (1) $s\vert_{\mathfrak{a}}=0$, (2) $s\vert_{\mathfrak{g}_{\gamma}}=0$ for all $\gamma\in \Phi^{+}_{\mathfrak{a}}\setminus \Upsilon$ and (3) $s(X_j)\in \{\pm 1\}$ (so there are $2^{r}$ open $AN$-orbits in $(\mathfrak{a}\oplus\mathfrak{n})^*$).  As $\beta_1$ is the highest root, it follows that $\beta_1$ is a long restricted root.  We end this section with a lemma which is useful later on. \\

\begin{lemma}
Let $\Omega$ be an open $AN$-orbit in $(\mathfrak{a}\oplus\mathfrak{n})^*$, then $\{h(X_1): h\in \Omega\}$ is contained in $\mathbb{R}^{*}_+$ or $\mathbb{R}^{*}_-$.

\end{lemma}

\begin{proof}

It is sufficient to notice that each $h\in\Omega$ is of the form $b.s$, where $b\in AN$ and $s$ is the unique element in $\Omega$ described previously. As $X_1$ is a highest root vector, it follows that for all $b\in AN$, $b.X_1\in \mathbb{R}^{*}_+.X_1$.
\end{proof}

\noindent \textbf{Remark}. In general, this lemma is false for $X_j$ with $j\neq 1$.

\section{Proof of theorem 3.1}

   From now on we assume the existence of an open $AN$-coadjoint orbit in $(\mathfrak{a}\oplus\mathfrak{n})^*$ (according to the remark (1) of section 3, otherwise,
   the theorem is automatically true). According to the previous section, $\beta_j$ is the restriction of a unique real root in $\Phi^{+}_{\mathfrak{h}}$ and $\beta_j$ is strongly orthogonal to $\beta_i$ for $1\leq i\neq j \leq r$, with $r=\dim(\mathfrak{a})$. Thus the process of Cayley transforms applied to $\mathfrak{h}$ allows us to see that $\mathfrak{t}=\mathfrak{h}_{k}\oplus \bigoplus^{r}_{j=1}\mathbb{R}(X_j+\theta(X_j))$ is a $\theta$-stable compact Cartan sub-algebra (notice that $\theta(X_j)\in \mathfrak{g}_{-\beta_j}$). Moreover under the identification $\mathfrak{t}\cong\mathfrak{t}^*$ of section 2, $Y_j:=X_j+\theta(X_j)$ is proportional to a non-compact root $\alpha_j$ with respect to the roots system $\Delta:=\Delta(\mathfrak{g}_{\mathbb{C}}, \mathfrak{t}_{\mathbb{C}})$. Especially, $\alpha_1$ is a long root, since $\beta_1$ is a long restricted root and the Cayley transforms preserve the length of roots. From now on, we will work on the compact Cartan sub-algebra $\mathfrak{t}$ constructed above.\\

Next we begin to prove the Theorem 3.1. It is clear that it is sufficient to prove it for $G$ simple connected:

\begin{proof}

We first prove the "$\Longleftarrow$" part.\\

Suppose that $\text{p}(G.f)$ is an open $AN$-orbit. Note that $\text{p}(G.f)=AN.\text{p}(K.f)$. Next it is clear that for each $k\in K$, we have $\text{p}(k.f)(X_1)=(k.f)(X_1)$. Hence we conclude that $\{(k.f)(X_1): k\in K\}$ is contained in $\mathbb{R}^{*}_+$ or $\mathbb{R}^{*}_-$. Further as $K.f\in \mathfrak{k}^*$ and $X_1=\frac{Y_1}{2}+\frac{X_1-\theta(X_1)}{2}$, we have $(k.f)(X_1)=(k.f)(\frac{Y_1}{2})$. This implies especially $\{\langle f, w.Y_1\rangle: w\in W_K\}$ is contained in $\mathbb{R}^{*}_+$ or $\mathbb{R}^{*}_-$. On the other hand, we have seen that $Y_1$ is proportional to a long non-compact root $\alpha_1$. Thus the "$\Longleftarrow$" part is a direct consequence of the lemma below.

\end{proof}

\begin{lemma}
 Let $G=K \exp{\mathfrak{p}}$ be a Cartan decomposition (with respect to the Cartan involution $\theta$) of a connected simple Lie group $G$ with Lie algebra $\mathfrak{g}$ . Suppose that $\mathfrak{t}$ is a $\theta$-stable compact Cartan sub-algebra. Let $f\in \mathfrak{t}^*\subset\mathfrak{g}^*$ be a strongly elliptic element such that $\mathfrak{t}\subset\mathfrak{g}(f)$. Suppose that there exists a long non-compact root $\beta$ in $\Delta=\Delta(\mathfrak{g}_{\mathbb{C}}, \mathfrak{t}_{\mathbb{C}})$, such that $\{\langle f, iw.\beta\rangle: w\in W_K\}$ is contained in $\mathbb{R}^{*}_+$ or $\mathbb{R}^{*}_-$. Then $f$ is holomorphic.
\end{lemma}

\begin{proof}
  Firstly, if the condition in the lemma is satisfied, then $\langle f, i\sum_{w\in W_K}w.\beta\rangle\neq 0$. Hence $i\sum_{w\in W_K}w.\beta\neq0$. But $i\sum_{w\in W_K}w.\beta$ is invariant under $W_K$. Thus it is in the center of $\mathfrak{k}$. This implies that the center of  $\mathfrak{k}$ is non trivial. Hence $\mathfrak{g}$ must be of Hermitian type. It follows that the $\text{Ad}$-representation of $K$ in $\mathfrak{p}_{\mathbb{C}}$ decomposes into two irreducible components. Moreover in this case, the 2 irreducible components $\mathfrak{p}_{\mathbb{C}}^+$, $\mathfrak{p}_{\mathbb{C}}^{-}$ are abelian. Then without loss of generality, we can assume our $\beta\in\Delta^{+}_{n}$, where $\Sigma_{\alpha\in \Delta_{n}^+} \mathfrak{g}_\alpha=\mathfrak{p}_{\mathbb{C}}^+$. However $\beta$ is a long root, thus an extreme weight for the  $\text{Ad}$-representation of $K$. Hence according to a Kostant's theorem, $\Delta_{n}^+$ is contained in the convex hull of $W_K.\beta$, $\text{conv}(W_K.\beta)$. Then we deduce that $\{\langle f, i\alpha\rangle: \alpha\in \Delta_{n}^+\}$ is contained in $\mathbb{R}^{*}_+$ or $\mathbb{R}^{*}_-$. Hence the lemma is proved.

\end{proof}

Now we want to prove the "$\Longrightarrow$" part of the Theorem 3.1. For this, we only need to treat the simple Lie groups of Hermitian type. But firstly we want to prove a general proposition for solvable Lie groups, then apply this proposition to our situation.

\subsection{Open coadjoint orbits for solvable Lie groups}

\begin{proposition}
Let $S$ be a connected solvable Lie group with Lie algebra $\mathfrak{s}$. Suppose that

\noindent (i) $\mathfrak{s}=\mathfrak{s}_1\oplus\mathfrak{s}_2$, where  $\mathfrak{s}_1$ is a Lie subalgebra and $\mathfrak{s}_2$ is an ideal of $\mathfrak{s}$.

\noindent (ii) $\mathfrak{s}_3\subset\mathfrak{s}_2$ is an abelian ideal of $\mathfrak{s}$, which verifies $[\mathfrak{s}_2, \mathfrak{s}_2]\subset \mathfrak{s}_3$ and $[\mathfrak{s}_2, \mathfrak{s}_3]=\{0\}$.

\noindent (iii) $\dim(\mathfrak{s}_3)=\dim(\mathfrak{s}_1)$.

\noindent (iv) There exists an open $S$-coadjoint orbit in $\mathfrak{s}^*$.

Let $\lambda\in \mathfrak{s}^*$ and $\lambda_3:=\lambda\vert_{\mathfrak{s}_3}$. Then
the coadjoint orbit $S.\lambda$ is open in $\mathfrak{s}^*$ if and only if $S.\lambda_3$ is an open orbit in $\mathfrak{s}_3^*$.
\end{proposition}

\begin{proof}
If $S.\lambda$ is open, it is obvious that $S.\lambda_3$ is open. Next we will prove "$\Longleftarrow$".

 Define ${\mathfrak{s}_3^*}':=\{\lambda\in \mathfrak{s}_3^*: S.\lambda \ \text{is open in} \ \mathfrak{s}_3^* \}$ and $\widetilde{\mathfrak{s}_3^*}:= \{\lambda_3\in \mathfrak{s}_3^*:\text{there exists a regular element }$ $ \lambda\in \mathfrak{s}_2^* \ \text{ such that} \ \lambda\vert_{\mathfrak{s}_3}=\lambda_3 \}$ (recall that an element $\lambda\in \mathfrak{s}_2^*$ is called regular, if the Lie algebra of its stabilizer $\mathfrak{s}_2(\lambda)$ is of minimal dimension). Then $\widetilde{\mathfrak{s}_3^*}$ is open and dense in $\mathfrak{s}_3^*$. On the other hand, $[\mathfrak{s}_2, \mathfrak{s}_2]\subset \mathfrak{s}_3$ and $\mathfrak{s}_2$ is an ideal. Thus we deduce that $\widetilde{\mathfrak{s}_3^*}$ is $S$-invariant and $\lambda_2\in\mathfrak{s}_2^*$ is regular if and only of $\lambda_2\vert_{\mathfrak{s}_3}\in\widetilde{\mathfrak{s}_3^*}$. Then the $S$-invariance and density of $\widetilde{\mathfrak{s}_3^*}$ imply that each open $S$-orbit of $\mathfrak{s}_3^*$ is contained in $\widetilde{\mathfrak{s}_3^*}$. In other words, we have ${\mathfrak{s}_3^*}'\subset\widetilde{\mathfrak{s}_3^*}$.

Since $[\mathfrak{s}_2, \mathfrak{s}_3]=0$, it is clear that for all $\lambda_2\in\mathfrak{s}_2^*$, we have $\mathfrak{s}_3\subset\mathfrak{s}_2(\lambda_2)$. Next we want to prove that for $\lambda_2$ regular in $\mathfrak{s}_2^*$, we have $\mathfrak{s}_3=\mathfrak{s}_2(\lambda_2)$. Actually, according to our assumption, we can take a $\widetilde{\lambda}\in \mathfrak{s}^*$ which lies in an open $S$-orbit. Denote $\widetilde{\lambda_2}:=\widetilde{\lambda}\vert_{\mathfrak{s}_2}$. Hence it is clear that $\mathfrak{s}_2(\widetilde{\lambda_2})=\mathfrak{s}_2\cap\mathfrak{s}_2^{\perp_{B_{\widetilde{\lambda}}}}$, where $\mathfrak{s}_2^{\perp_{B_{\widetilde{\lambda}}}}$ is the orthogonal of $\mathfrak{s}_2$ in $\mathfrak{s}$ with respect to the Kirillov-Kostant-Souriau symplectic form $B_{\widetilde{\lambda}}=\widetilde{\lambda}([,])$. However, we have $\dim(\mathfrak{s}_2^{\perp_{B_{\widetilde{\lambda}}}})=\dim(\mathfrak{s})-\dim(\mathfrak{s}_2)=\dim(\mathfrak{s}_3)$. Then we have $\mathfrak{s}_2(\widetilde{\lambda_2})=\mathfrak{s}_3$. Hence $\mathfrak{s}_3=\mathfrak{s}_2(\lambda_2)$, for all $\lambda_2$ regular in $\mathfrak{s}_2^*$.

Now assume $\lambda\in\mathfrak{s}^*$ such that $\lambda_3:=\lambda\vert_{\mathfrak{s}_3}\in {\mathfrak{s}_3^*}'$, i.e., $S.\lambda_3$ is open. Let $\lambda_2:=\lambda\vert_{\mathfrak{s}_2}$. Then according to what we have seen, $\lambda_2$ is regular. Now since $S.\lambda_3$ is open and $\dim(\mathfrak{s})-\dim(\mathfrak{s}_3)=\dim(\mathfrak{s}_2)$, we have $\dim(\mathfrak{s}(\lambda_3))=\dim(\mathfrak{s}_2)$. But it is clear that $\mathfrak{s}_2\subset\mathfrak{s}(\lambda_3)$. Thus $\mathfrak{s}_2=\mathfrak{s}(\lambda_3)$. Then we deduce that $\mathfrak{s}(\lambda)\subset\mathfrak{s}_2(\lambda_2)$. But we have proved $\mathfrak{s}_2(\lambda_2)=\mathfrak{s}_3$. Hence we deduce that $\mathfrak{s}(\lambda)$ equals the orthogonal of $\mathfrak{s}.\lambda_3 \subset\mathfrak{s}_3^*$. However, $\mathfrak{s}.\lambda_3=\mathfrak{s}_3^*$. Hence $\mathfrak{s}(\lambda)=0$. Then "$\Longleftarrow$" is proved.

\end{proof}

\noindent \textbf{Remark}. If $\mathfrak{s}_2=\mathfrak{s}_3$, then we can drop the assumption that there exists an $S$-open orbit in $\mathfrak{s}$. This can be easily seen from the proof.\\

 Let $G$ be simple of Hermitian type. Then the restricted roots system $\Phi_{\mathfrak{a}}$ is contained in $\{\pm\frac{1}{2}(\beta_i+\beta_j)\}_{1\leq i, j\leq r }\cup \{\pm\frac{1}{2}(\beta_i-\beta_j)\}_{1\leq i < j\leq r }\cup \{\pm\frac{1}{2}\beta_i\}_{1\leq i\leq r }$, where $r=\dim(\mathfrak{a})$. Notice that the terms "$\frac{1}{2}\beta_i$" might not appear in $\Phi_{\mathfrak{a}}$. We denote the ideals of $\mathfrak{a}\oplus\mathfrak{n}$, $\mathfrak{n}_3:=\bigoplus_{1\leq i, j\leq r}\mathfrak{g}_{\frac{1}{2}(\beta_i+\beta_j)}$ and $\mathfrak{n}_2:=\mathfrak{n}_3\oplus \bigoplus_{1\leq i\leq t}\mathfrak{g}_{\frac{1}{2}\beta_i}$. Then we have $\mathfrak{n}_3\subseteq \mathfrak{n}_2$ and $\mathfrak{a}\oplus\mathfrak{n}=\mathfrak{n}_2\oplus\mathfrak{n}_1$, where $\mathfrak{n}_1=\mathfrak{a}\oplus\bigoplus_{1\leq i < j\leq r }\mathfrak{g}_{\frac{1}{2}(\beta_i-\beta_j)}$. Hence the conditions of the previous proposition are satisfied: we replace "$\mathfrak{s}$" by $\mathfrak{a}\oplus\mathfrak{n}$ and "$\mathfrak{s}_i$" by $\mathfrak{n}_i$. Actually this can be easily seen for instance by the fact that there is a "$J$-algebra" structure in $\mathfrak{a}\oplus\mathfrak{n}$. Hence we have

\begin{corollary}
 Let $\lambda\in (\mathfrak{a}\oplus\mathfrak{n})^*$ and $\lambda_3:=\lambda\vert_{\mathfrak{n}_3}$. Then $AN.\lambda$ is an open $AN$-coadjoint orbit in $(\mathfrak{a}\oplus\mathfrak{n})^*$ if and only if $AN.\lambda_3$ is an open $AN$-orbit in $\mathfrak{n}_3^*$.
\end{corollary}

\subsection{strongly elliptic and holomorphic coadjoint orbits}

Recall that we want to show: if $f\in\mathfrak{g}^*$ is  strongly elliptic and holomorphic, then $\text{p}(\mathcal{O}_{f})$  is an open $AN$-coadjoint orbit.
Corollary 5.1 tells us that it is sufficient to show that $\text{p}_1(\mathcal{O}_{f})$ is an open $AN$-orbit in $\mathfrak{n}_3^*$, where $\text{p}_1: \mathfrak{g}^*\longrightarrow \mathfrak{n}_3^*$ is the natural projection.

Firstly, we translate it into the adjoint picture. Identify $\mathfrak{g}$ with $\mathfrak{g}^*$ via the inner product $\langle .,.\rangle$. Here $\langle X, Y\rangle=-\mathcal{K}(X,\theta(Y))$ for $X,Y\in \mathfrak{g}$. Then for $x\in G$ and $\mathfrak{g}^*\ni h=\langle., X_h\rangle$, we have $\text{Ad}^*(x).h=\langle.,\text{Ad}(\Theta(x)).X_h\rangle$. Thus we still have $\text{Ad}^*(G).h\cong \text{Ad}(G).X_h$. Denote $\text{pr}_{\mathfrak{n}_3}$ the orthogonal projection of $\mathfrak{g}$ onto $\mathfrak{n}_3$ with respect to $\langle .,.\rangle$.  Then we have the following lemma.

\begin{lemma}
The following diagram is commutative.

   \begin{tikzpicture}[scale=2.5]
   \node (A) at (0,1) {$\mathfrak{g}$};
   \node (B) at (1,1) {$\mathfrak{g}^*$};
   \node (C) at (0,0) {$\mathfrak{n}_3$};
   \node (D) at (1,0) {$\mathfrak{n}_3^*$};
   \path[->,font=\scriptsize,>=angle 90]
   (A) edge node[above]{$\cong$} (B)
    (A) edge node[below]{$\langle.,.\rangle$} (B)
   (A) edge node[left]{$\text{pr}_{\mathfrak{n}_3}$} (C)
   (B) edge node[right]{$\text{p}_1$} (D)
   (C) edge node[below]{$\langle.,.\rangle_{\mathfrak{n}_3\times \mathfrak{n}_3}$} (D)
   (C) edge node[above]{$\cong$} (D);
   \end{tikzpicture}

\end{lemma}

\begin{proof}
Let $h\in \mathfrak{g}^*$ with $ h=\langle., X_h\rangle$. For any $Y\in\mathfrak{n}_3$, $h(Y)=\text{p}_1(h)(Y)=\langle Y, X_h\rangle=\langle Y, \text{pr}_{\mathfrak{n}_3}(X_h)\rangle$. This completes the proof.
\end{proof}

  Fix $f\in\mathfrak{t}^*$ a strongly elliptic and holomorphic element which corresponds to $X_f\in\mathfrak{t}$. Recall that $\mathfrak{t}$ always denotes the $\theta$-stable compact Cartan sub-algebra which is constructed at the beginning of the section. Let $X_i\in\mathfrak{g}_{\beta_i}$ such that $\langle f, X_i+\theta(X_i)\rangle>0$ (*). Let $\mathfrak{n}_c:=\bigoplus_{i<j}\mathfrak{g}_{\frac{1}{2}(\beta_i-\beta_j)}$ and $N_c:=\exp(\mathfrak{n}_c)$.

\begin{lemma}
 $\text{p}_1(\text{Ad}(AN)^*.f)$ corresponds to the subset $\text{Ad}(\theta(N_c))\sum^{r}_{j=1}\mathbb{R}^{*}_+X_j$ in $\mathfrak{n}_3$.

\end{lemma}

\begin{proof}

We can write $f:= X_f=\sum^{r}_{j=1}c_j(X_j+\theta(X_j))+X_0\in\mathfrak{t}$, where $X_j\in\mathfrak{g}_{\beta_j}$ is the same as the ones in (*) and $c_j>0$ and $X_0\in\mathfrak{m}$. Here $\mathfrak{m}$ is the centralizer of $\mathfrak{a}$ in $\mathfrak{k}$.

Now let $a\in A$, $n\in N$ and $Y\in\mathfrak{n}_3$. Then $\text{Ad}^*(an)f(Y)=f(\text{Ad}(an)^{-1}Y)=\langle X_f, \text{Ad}(an)^{-1}Y\rangle$.  On the other hand $\text{Ad}(an)^{-1}Y\in \mathfrak{n}_3$ and $\theta(X_j)$ and $X_0$ are orthogonal to $\mathfrak{n}_3$ (actually even to $\mathfrak{n}$). Then $\text{Ad}^*(an)f(Y)=\langle\sum^{r}_{j=1}c_jX_j,\text{Ad}(an)^{-1}Y\rangle$. Hence we deduce that $$\text{p}_1(\text{Ad}(AN)^*.f)\cong\text{pr}_{\mathfrak{n}_3}(\text{Ad}(A\theta(N))\sum^{r}_{j=1}c_jX_j)=\text{pr}_{\mathfrak{n}_3}(\text{Ad}(\theta(N))\sum^{r}_{j=1}\mathbb{R}^{*}_+X_j).$$ However $N=N_3.N_{\frac{1}{2}}N_c$ with $N_3:=\exp(\mathfrak{n}_3)$ and $N_{\frac{1}{2}}:=\exp(\bigoplus_{1\leq j\leq r}\mathfrak{g}_{\frac{1}{2}\beta_j})$. Then we have $$\text{pr}_{\mathfrak{n}_3}(\text{Ad}(\theta(N))\sum^{r}_{j=1}\mathbb{R}^{*}_+X_j)=\text{pr}_{\mathfrak{n}_3}(\text{Ad}(\theta(N_3))\text{Ad}(\theta(N_{\frac{1}{2}})).\text{Ad}(\theta(N_c))\sum^{r}_{j=1}\mathbb{R}^{*}_+X_j).$$ Nevertheless, it is clear that for any $Y\in\mathfrak{n}_3$, $\text{pr}_{\mathfrak{n}_3}(\text{Ad}(\theta(N_3))\text{Ad}(\theta(N_{\frac{1}{2}})).Y)=Y$. Then the proof follows.

\end{proof}

  It is known that $\mathfrak{n}_3$ carries the structure of an Euclidean Jordan-algebra. Let $\Omega^+$ be (up to sign) the associated open convex cone.
  Recall the construction of $\Omega^+$. For that let $\mathfrak{g}_0:=\mathfrak{a}\oplus\mathfrak{m}\oplus\bigoplus_{i\neq j}\mathfrak{g}_{\frac{1}{2}(\beta_i-\beta_j)}$ with $\mathfrak{m}$ the centralizer of $\mathfrak{a}$ in $\mathfrak{k}$. Let $G_0:=\exp(\mathfrak{g}_0)$. Then
  $$\Omega^+=\text{Ad}(G_0)\sum^{r}_{j=1}X_j=\text{Ad}(A\theta(N_c))\sum^{r}_{j=1}X_j=\text{Ad}(\theta(N_c))(\sum^{r}_{j=1}\mathbb{R}^{*}_+X_j).$$

  Then we have the following.
  \begin{corollary}
  $\text{p}_1(\text{Ad}(AN)^*.f)$ corresponds to $\Omega^+$.
 \end{corollary}

Next we will prove our main theorem based on the fine geometry of convex cones in the simple Lie algebra of Hermitian type.

Thus let $\Delta_n^+$ be one of the two holomorphic subsets of non compact roots.

Define $c_{\max}:=\{X\in \mathfrak{t}: \forall \alpha\in\Delta_n^+, i\alpha(X)>0\}$. Then $X_f\in\pm c_{\max}$.   It is known that $C_{\max}:=\text{Ad}(G)c_{\max}$ is a proper maximal $\text{Ad}(G)$-invariant open convex cone in $\mathfrak{g}$ (see \cite{Ne}). Without loss of generality, we can assume $\mathcal{O}_{X_f}:=\text{Ad}(G).X_f\subset C_{\max}$.  recall that $X_j\in\mathfrak{g}_{\beta_j}$
 are those in (*). Since $X_1$ is a highest weight vector for $\text{Ad}$-representation of $G$ on $\mathfrak{g}$,  we have the
 following characterization of $C_{\max}$ due to  Paneitz-Vinberg (see Theorem 2.1.21 in \cite{HO}).

  $$C_{\max}=\{X\in\mathfrak{g}: \langle X, \text{Ad}(g).X_1\rangle>0, \forall g\in G\}.$$

  Hence as each $X_j$ is conjugate to $X_1$ via Weyl group (up to a positive scalar), we deduce the following

   \begin{corollary}
      For each $Y\in C_{\max}$, we have $\langle Y, \Omega^+ \rangle>0$.

   \end{corollary}

   Now in order to conclude $\text{p}_1(\mathcal{O}_f)=\Omega^+$ (then our theorem is proved) we prove the following.

   \begin{corollary}
          $\text{pr}_{\mathfrak{n}_3}(C_{\max})= \Omega^+ $.

   \end{corollary}

   \begin{proof}

   Firstly, $\text{pr}_{\mathfrak{n}_3}(C_{\max})\supseteq \Omega^+ $ follows from Corollary 5.2 and  Lemma 5.2.

   Next it is known that the closure of $\Omega^+$, $\overline{\Omega^+}$ is self-dual (see \cite{FK}): i.e., $X\in\overline{\Omega^+}$
    if and only if $\langle X, \overline{\Omega^+} \rangle\geq0$. Then according to the previous corollary, we have
    $\text{pr}_{\mathfrak{n}_3}(C_{\max})\subset\overline{\Omega^+}$. But $C_{\max}$ is open and $\text{pr}_{\mathfrak{n}_3}$ is an open map. Hence we deduce that $\text{pr}_{\mathfrak{n}_3}(C_{\max})\subseteq \Omega^+ $.

   \end{proof}

\noindent \textbf{Remark}. Since $G$ is simple of Hermitian type, the set of strongly elliptic and holomorphic elements has two connected components $\pm\Psi^+$ (actually $\Psi^+\cong C_{\max}$). Since $\Psi^+$ is union of strongly elliptic and holomorphic $G$-orbits, a simple topological argument implies that $\text{p}(\Psi^{+})=\Omega_{+}$, where $\Omega_{+}$ is an open $AN$-orbit in $(\mathfrak{a}\oplus\mathfrak{n})^*$. In other words, among many open $AN$-orbits in $(\mathfrak{a}\oplus\mathfrak{n})^*$, there are only two and exactly two opposite open orbits onto which the cone of strongly elliptic and holomorphic elements  in $\mathfrak{g}^*$ are projected.

\subsection*{Acknowledgements}  I would particularly like to thank Prof. Kr\"{o}tz for his crucial help for convex cone theory which is essential for our general proof. I also thank Prof. Duflo with whom I had useful discussions for communicating Carmona's unpublished paper to us. I would like to thank Prof. Torasso for useful discussions, some of his ideas are helpful for our work. Finally, my thanks goes to Prof. Hilgert for comments and discussions on a preliminary version of the paper.

\end{document}